\documentclass[letterpaper, 10 pt, conference]{ieeeconf}

\usepackage{mathrsfs}
\usepackage{graphics} 
\usepackage{amsmath} 
\usepackage{amssymb}  

\usepackage{amsthm}
\usepackage{cite}
\usepackage{bm}
\usepackage{acronym}
\usepackage{paralist}
\usepackage[colorinlistoftodos,prependcaption,textsize=small]{todonotes}
\usepackage{color}
\usepackage{tikz}
\usepackage{hyperref}
\usepackage{graphicx}
\usepackage{soul}
\usepackage{mathtools}
\usepackage{siunitx}
\usepackage{empheq}
\usepackage{accents}
\usepackage{mdframed}
\renewcommand{\theenumi}{\roman{enumi}}

\renewtagform{default}{\normalsize(}{)}
\usetagform{default}

\makeatletter
\hypersetup{colorlinks=true}
\AtBeginDocument{\@ifpackageloaded{hyperref}
  {\def\@linkcolor{blue}
  \def\@anchorcolor{red}
  \def\@citecolor{red}
  \def\@filecolor{red}
  \def\@urlcolor{red}
  \def\@menucolor{red}
  \def\@pagecolor{red}
\begingroup
  \@makeother\`%
  \@makeother\=%
  \edef\x{%
    \edef\noexpand\x{%
      \endgroup
      \noexpand\toks@{%
        \catcode 96=\noexpand\the\catcode`\noexpand\`\relax
        \catcode 61=\noexpand\the\catcode`\noexpand\=\relax
      }%
    }%
    \noexpand\x
  }%
\x
\@makeother\`
\@makeother\=
}{}}
\makeatother

\newtheorem{Theorem}{Theorem}
\newtheorem{Lemma}{Lemma}
\newtheorem{Problem}{Problem}
\newtheorem{Remark}{Remark}

\newtheorem{Assumption}{Assumption}
\newtheorem{Definition}{Definition}

\DeclareMathOperator{\R}{\mathbb R}

\DeclareMathOperator*{\argmin}{arg\,min}

\renewcommand{\P}{\mathbb P}
\renewcommand{\S}{\mathcal S}
\newcommand{\X}{\mathcal X}
\newcommand{\T}{\mathcal T}
\renewcommand{\L}{\mathcal L}
\renewcommand{\d}{\mathrm d}

\newcommand{\erf}{\mathrm{erf}}

\newcommand{\dt}{\text{d}t}

\newcommand{\bb}{\boldsymbol}
\newcommand{\bx}{\boldsymbol{x}}
\newcommand{\bxh}{\hat{\boldsymbol{x}}}

\setuptodonotes{inline}

\IEEEoverridecommandlockouts

\def\BibTeX{{\rm B\kern-.05em{\sc i\kern-.025em b}\kern-.08em
    T\kern-.1667em\lower.7ex\hbox{E}\kern-.125emX}}

\begin{document}

\title{\LARGE \bf Risk-Aware Fixed-Time Stabilization of Stochastic Systems under Measurement Uncertainty}

\author{Mitchell Black$^1$ \and Georgios Fainekos$^1$ \and Bardh Hoxha$^1$ \and Dimitra Panagou$^2$
\thanks{$^1$Toyota North America Research \& Development, 1555 Woodridge Ave, Ann Arbor, MI 48105, USA; \texttt{\{first.last\}@toyota.com}.}
\thanks{$^2$Dept. of Robotics and Dept. of Aerospace Engineering,  Univ. of Michigan, Ann Arbor, MI 48109, USA; \texttt{\{dpanagou\}@umich.edu}.}
}
\maketitle  


\begin{abstract}\label{sec: abstract}
This paper addresses the problem of risk-aware fixed-time stabilization of a class of uncertain, output-feedback nonlinear systems modeled via stochastic differential equations. First, novel classes of certificate functions, namely risk-aware fixed-time- and risk-aware path-integral-control Lyapunov functions, are introduced. Then, it is shown how the use of either for control design certifies that a system is both stable in probability and probabilistically fixed-time convergent (for a given probability) to a goal set. That is, the system trajectories probabilistically reach the set within a finite time, independent of the initial condition, despite the additional presence of measurement noise. These methods represent an improvement over the state-of-the-art in stochastic fixed-time stabilization, which presently offers bounds on the settling-time function in expectation only. The theoretical results are verified by an empirical study on an illustrative, stochastic, nonlinear system and the proposed controllers are evaluated against an existing method. Finally, the methods are demonstrated via a simulated fixed-wing aerial robot on a reach-avoid scenario to highlight their ability to certify the probability that a system safely reaches its goal.
\end{abstract}


\section{Introduction}\label{sec.intro}

Certifying the stability and/or convergence properties of a complex dynamical system is an important step in verification and control design, and often requires an accurate system model. While much attention has been paid to developing tools for system identification, even high-fidelity models may fail to capture the exogenous perturbations and random phenomena characteristic in many real-world systems. Modeling a system as a stochastic differential equation (SDE) provides a way to methodologically account for uncertain behavior, and is a common choice for systems from aerial robots and ground rovers to financial markets and climate models. As such, it is necessary to design controllers that confer the requisite stability certificates to stochastic systems.

In this regime, however, stability certificates are not absolute but rather probabilistic in nature. Traditionally, control Lyapunov functions (CLFs) have served as a mechanism for certifying asymptotic stability of an equilibrium point \cite{sontag1983lyapunov} (or asymptotic convergence to a goal set) for deterministic systems. Analogous theories for stochastic and switched stochastic systems (e.g., \cite{florchinger1997feedback}, \cite{dimarogonas2004lyapunov}) have also been developed, but they, like their deterministic cousins, only provide guarantees in the limit as time tends to infinity. Many systems operate over a finite time interval, or are subject to temporal specifications and spatiotemporal constraints, and therefore may require the stronger notions of finite- (FTS) or fixed-time stability (FxTS) introduced by \cite{bhat2000finite} and \cite{polyakov2011nonlinear} respectively. Whereas FTS of an equilibrium implies Lyapunov stability and convergence within a bounded finite time, a FxTS equilibrium (or goal set) is reached within a finite time independent of the initial condition. Since the development of stochastic CLFs, advancements have been made in theory for stochastic FTS both with state- \cite{chen2010finite,yin2011finite} and output-feedback \cite{Zha2014Finite}, but only in recent years has stochastic FxTS been addressed \cite{Yu2019Fixed} (and only for state-feedback or output-feedback for systems in strict-feedback form \cite{li2023prescribed}).

Yet despite these advances, neither the existing theory for FTS nor FxTS in stochastic control characterizes the level of associated risk, i.e., the probability of failing to reach the equilibrium or goal set, beyond specifying the outcome in expectation. Works related to this open problem include those studying probabilistic reachable sets. For example, \cite{Hewing2018Stochastic} defines probabilistic $n$-step reachable sets and \cite{abate2008probabilistic} computes maximal initial sets to bound the probability that the system trajectories reach some final set, though these ideas differ from finite- and fixed-time attractivity.
In verification and control design for many real applications, certifying that a spatiotemporal constraint shall be met with a probability of $0.5$ may be unacceptable; a much lower risk may need to be met. 

In control design, the notion of risk-awareness is not new. Recent works have proposed risk-aware controllers in the context of probabilistically avoiding constraint violations, e.g., \cite{Yaghoubi2021RiskKalman,singletary2022safe}, as well as our recent work aimed at reducing conservatism \cite{black2023safety}. Another class of methods are aimed at satisfying risk metrics, e.g., the increasingly popular conditional value-at-risk, over the distribution of desirable or undesirable outcomes \cite{hakobyan2019risk,ahmadi2021risk}, while others synthesize such risk measures into objective functions for optimization-based control laws \cite{yin2023risk,wei2023moving}. 

For the class of stochastic, nonlinear systems under consideration in this paper neither has any work addressed the matter of arbitrarily probabilistic FxTS of an equilibrium point, nor has the problem of output-feedback stabilization in fixed-time been solved. As such, this paper makes the following contributions: 
\begin{enumerate}[1)]
    \item it introduces novel classes of risk-aware fixed-time CLFs (RA-FxT-CLFs) and risk-aware path-integral CLFs (RA-PI-CLFs) for the fixed-time stabilization of a generic class of stochastic, nonlinear systems to a goal set under the additional effect of stochastic measurement noise; 
    \item it proves how the use of either RA-FxT-CLFs or RA-PI-CLFs for control design certifies that their associated goal set is probabilistically FxTS with probability $p_g$, i.e., that the system trajectories reach the goal set within the given fixed-time with probability $p_g$; and 
    \item it highlights the efficacy of RA-FxT-CLF- and RA-PI-CLF-based control laws in a comparative study against a control law derived from the stochastic FxT-CLFs introduced in \cite{Yu2019Fixed} on a demonstrative nonlinear system.
\end{enumerate}

\section{Preliminaries and Problem Formulation}\label{sec.prelims}
The uniform distribution supported by $a$ and $b$ is $\mathbb{U}[a,b]$. 
The sets of real numbers, real $n$-valued vectors, and real $n\times n$-valued matrices are $\R$, $\R^n$, and $\R^{n\times n}$ respectively. The trace of a matrix $\bb{M} \in \R^{n \times n}$ is $\textrm{Tr}(\bb{M})$.
$\mathcal{C}^n$ denotes the set of $n$-times continuously differentiable functions.
A bolded $\bx_t$ denotes a vector stochastic process at time $t$.
The Gauss error function is $\textrm{erf}(z) = \frac{2}{\sqrt{\pi}}\int_0^ze^{-t^2}\d t$, and $\textrm{erf}^{-1}(\cdot)$ is its inverse. 
The local Lipschitz constant of a function $\phi: \R^n \mapsto \R^m$ on a domain $\mathcal{D}$ is denoted $\lambda_\phi(\mathcal{D})$, i.e., $\|\phi(\bx) - \phi(\bx')\| \leq \lambda_\phi(\mathcal{D})\|\bx - \bx'\|$ for all $\bx,\bx' \in \mathcal{D}$. For a bounded function $\phi$, its Euclidean norm bound is denoted $b_\phi(\mathcal{D}) \triangleq \sup_{\bx \in \mathcal{D}}\|\phi(\bx)\|$. For the above, the domain $\mathcal{D}$ may be omitted when context is clear. For a closed set $\S$, $\|\bx\|_{\S}$ denotes the Euclidean distance of $\bx$ from $\S$. The set of extended class-$\mathcal{K}$ functions is denoted $\mathcal{K}_\infty$.

\subsection{Mathematical Preliminaries}
This paper is concerned with system operation over a finite time interval $\mathcal{T} = \{t \in \R_{\geq 0}: t \leq T\}$, where $T < \infty$. Consider $1$- and $p$-valued standard Wiener processes (i.e., Brownian motions) $w: \R_{\geq 0} \mapsto \R$ and $\bb{w}: \R_{\geq 0} \mapsto \R^p$ respectively defined over the complete probability space $(\Omega, \mathcal{F}, \P)$ for sample space $\Omega$, $\sigma_x$-algebra $\mathcal{F}$ over $\Omega$, and probability measure $\P: \mathcal{F} \mapsto [0,1]$. The following are required in the proofs of this paper's main results.
\begin{Lemma}[Level Crossing]\label{lem.wiener_crossing_probability}
    Given $a>0$, the probability that $w_t < a$, $\forall t \in \T$, is given by
    \begin{equation*}
        \P\left\{\sup_{t \in \T}w_t < a \right\} = \mathrm{erf}\left(\frac{a}{\sqrt{2T}}\right).
    \end{equation*}
\end{Lemma}
\begin{proof}
    Given in Appendix \ref{app.wiener_crossing_probability}.
\end{proof}
\begin{Lemma}[Multivariate It$\hat{\mathrm{o}}$ Isometry]\label{lem.multidim_ito_isometry}
    Let $\bb{v}: \T \mapsto \R^{q}$ be a stochastic process adapted to the filtration of $\bb{w}$. Then,
    \begin{equation*}
        \mathbb{E}\left[\left(\int_0^t \bb{v}_s^\top\mathrm{d}\bb{w}_s\right)^2\right] = \mathbb{E}\left[\int_0^t\|\bb{v}_s\|^2\mathrm{d}s\right].
    \end{equation*}
\end{Lemma}
\begin{proof}
    Given in Appendix \ref{app.proof_multidim_ito_isometry}.
\end{proof}

In the remainder, consider systems whose dynamics may be described by the following class of nonlinear, stochastic differential equations (SDE),
\begin{subequations}\label{eq.stochastic_system}
\begin{align}
    \d\bx_t &= f(\bx_t,\bb{u}_t)\d t + \sigma_x(\bx_t)\d\bb{w}_t, \label{eq.stochastic_dynamics} \\
    \d\bb{y}_t &= h(\bx_t)\d t + \sigma_y(\bx_t)\d\bb{v}_t, \label{eq.measurement_model}
\end{align}
\end{subequations}
where $\bx \in \mathcal{X} \subset \R^n$ denotes the state, $\bb{u} \in \mathcal{U} \subseteq \R^m$ the control input, $\bb{y} \in \R^p$ the measurable output, where $\bb{w} \in \R^n$, $\bb{v} \in \R^p$ are independent, standard $n$- and $p$-valued Wiener processes adapted to $\mathcal{F}$, and where $f: \mathcal{X} \times \mathcal{U} \mapsto \R^n$, $\sigma_x: \X \mapsto \R^{n \times n}$, $h: \X \mapsto \R^p$, and $\sigma_y: \X \mapsto \R^{p \times p}$ are known and continuous in their arguments. Let $\varphi_t(\bb{x}_0,\bb{u})$ denote the solution to \eqref{eq.stochastic_dynamics} at time $t$ under the effect of the input $\bb{u}$ beginning from the initial condition $\bb{x}_0$.
Adapted from \cite[Def. 7.3.1]{Oksendal2003Stochastic} to include control, the generator of a stochastic process is analogous to the Lie derivative for deterministic systems in that it characterizes the derivative of a function $\phi$ over the trajectories of \eqref{eq.stochastic_dynamics} in expectation.
\begin{Definition}\label{def.generator}
    The (infinitesimal) generator of $\bx_t$ is
    \begin{equation*}
        \L^{\bx}\phi(\bx, \bb{u}) \triangleq \lim_{t \rightarrow 0^+}\frac{\mathbb{E}\left[\phi(\bb{\varphi}_t(\bx, \bb{u})) \; | \; \bx_0 = \bx\right] - \phi(\bx)}{t},
    \end{equation*}
    where $\phi: \R^n \mapsto \R$ belongs to  $\mathcal{D}_\L$, the set of all functions such that the limit exists for all $\bx \in \R^n$.
\end{Definition}
By \cite[Thm. 7.3.3]{Oksendal2003Stochastic}, for a twice continuously differentiable function $\phi$ with compact support, i.e., $\phi \in \mathcal{C}_0^2(\X) \subset \mathcal{D}_\L$, 
\begin{equation*}
    \L^{\bx}\phi(\bx, \bb{u}) = \frac{\partial \phi}{\partial \bx}f(\bx, \bb{u}) + \frac{1}{2}\textrm{Tr}\left(\sigma_x(\bx)^\top\frac{\partial^2 \phi}{\partial \bx^2}\sigma_x(\bx)\right),
\end{equation*}
the Lebesgue integral over time of which, for sample path $\omega \in \Omega$, is denoted
\begin{equation}\label{eq.integrated_generator}
        I_\phi(t, \omega) \triangleq \int_0^t\L^{\bx}\phi(\bx_s(\omega),\bb{u}_s(\omega))\mathrm{d}s.
\end{equation}
For the system given by \eqref{eq.stochastic_system}, a state observer and suitable, estimate-feedback controller may be defined by
\begin{subequations}\label{eq.observer_controller_system}
\begin{align}
    \d \bxh_t &= f(\bxh_t, \bb{u}_t)\d t + \bb{K}_t\big(\d\bb{y}_t - h(\bxh_t)\d t\big), \label{eq.state_observer}\\
    \bb{u}_t &= k(t, \bxh_t), \label{eq.controller}
\end{align}
\end{subequations}
where $\bb{K}$ is the observer gain and $k: \mathcal{T} \times \R^n$ is piecewise-continuous in $t$ and locally Lipschitz in the observer state $\bxh \in \R^n$. Note that \eqref{eq.state_observer} describes a broad class of stochastic state observers, including families of Kalman-Bucy filters, e.g., EKBF \cite{reif1999stochastic}, UKBF \cite{Xu2008UKF}, etc. For many such observers, it may be shown that the error is bounded in probability under certain conditions (i.e., detectability of $\left(\frac{\partial f}{\partial \bx}(\bx,\bb{u}), h(\bx, \bb{u})\right)$, see \cite{reif1999stochastic, Yaghoubi2021RiskKalman}), which motivates the following assumption.
\begin{Assumption}\label{ass.bounded_observer_error}
    There exists $\epsilon_0 > 0$ such that for all $\|\bx_0 - \bxh_0\| \leq \epsilon_0$, there exist $\delta \in (0,1)$, $\epsilon(\delta) > 0$ such that
    \begin{equation*}
        \P\left\{\sup_{t \geq 0}\|\bx_t - \bxh_t\| \leq \epsilon(\delta)\right\} \geq 1 - \delta.
    \end{equation*}
\end{Assumption}
As noted in \cite{Jahanshahi2020Partial}, stochastic simulation functions introduced by \cite{julius2008probabilistic} may be used to determine the relationship between $\epsilon$ and $\delta$. It is therefore assumed in the remainder that $\delta$, $\epsilon$, and $\epsilon_0$ are known, and that $\bxh_0$ satisfies $\|\bx_0 - \bxh_0\| \leq \epsilon_0$.

\subsection{Problem Formulation}
This paper considers the finite-time stabilization (in probability) of the trajectories of \eqref{eq.stochastic_dynamics} to a neighborhood of the origin (assumed to be an equilibrium), i.e., to a goal set $\S_g$ that contains the origin and is defined by
\begin{align}
    \S_g &= \{\bx \in \mathcal{X} \mid V(\bx) \leq 0\}, \label{eq.goal_set} 
\end{align}
for a function $V \in \mathcal{C}^2: \mathcal{X} \mapsto \R$ satisfying
\begin{align}
    \alpha_1(\|\bx\|_{\S_g}) \leq V(\bx) \leq \alpha_2(\|\bx\|_{\S_g}), \label{eq.V_pos_decrescent} \\
    0 < V(\bx_0) \leq \gamma_V, \; \forall \bx_0 \in \mathcal{X}_0 \subset \X, \label{eq.V0_r}
\end{align}
for $\alpha_1,\alpha_2 \in \mathcal{K}_\infty$ and known $\gamma_V < \infty$. The following introduce relaxed notions of finite- and fixed-time stable (in probability) sets as compared to \cite{yin2011finite} and \cite{Yu2019Fixed} in that finite-time attractivity is permitted to hold for a given probability.
\begin{Definition}\label{def.stochastically_FTS}
    Given $p_g \in (0,1)$, the set $\S_g$ is \textbf{locally finite-time stable with probability p$_{\bb{g}}$} ($p_g$-FTS) for the system \eqref{eq.stochastic_system} under control policy $\pi_{\bb{u}}(\cdot) \in \mathcal{U}$ if, $\forall \bx_0 \in \mathcal{D} \subseteq \X$, $\exists T_s(\bx_0) < \infty$ such that
    \begin{enumerate}
        \item for every pair $(p, \chi(p))$, $p \in (0,1)$, $\chi = \chi(p) > 0$, there exists $\psi = \psi(\chi,p) > 0$ such that 
        \small{
        \begin{equation*}
            \|\bx_0\|_{\S_g} \leq \psi \implies \P\left\{\sup_{t \geq 0}\|\bb{\varphi}_t(\bx_0,\pi_{\bb{u}})\|_{\S_g} \leq \chi\right\} \geq 1 - p,
        \end{equation*}}\normalsize
        \item $\P\left\{\lim_{t \rightarrow T_s(\bx_0)}\|\bb{\varphi}_t(\bx_0, \pi_{\bb{u}})\|_{\S_g} = 0\right\} \geq p_g$.
    \end{enumerate}
\end{Definition}
The first element of the above requires that $\S_g$ is stable in probability for \eqref{eq.stochastic_dynamics}, and the second requires that $\S_g$ is locally finite-time attractive for \eqref{eq.stochastic_dynamics} with probability $p_g$ for bounded settling time function $T_s: \mathcal D \mapsto (0,\infty)$. 

\begin{Definition}
    Given $p_g \in (0,1)$, the set $\S_g$ is \textbf{fixed-time stable with probability p$_{\bb{g}}$} ($p_g$-FxTS) for the system \eqref{eq.stochastic_system} under control policy $\pi_{\bb{u}}(\cdot) \in \mathcal{U}$ if $\exists T_{max} \in (0, \infty)$ such that
    \begin{enumerate}
        \item $\S_g$ is globally finite-time stable with probability $p_g$, and
        \item $\mathbb{P}\left\{T(\bx_0) \leq T_{max} \right\} \geq p_g$, $\forall \bx_0 \in \X$. 
    \end{enumerate}
\end{Definition}
In contrast to the existing notion of stochastic fixed-time stability (see \cite{Yu2019Fixed}), the above definition specifies a condition on the probability of the settling time function rather than its expectation. This may be more appropriate in cases where the settling time function is non-Gaussian, or where controllers must trade-off risk, i.e., the probability of violating some hard system constraint (like safety), with reward, e.g., the likelihood of reaching a goal set within a prescribed time.

In what follows, the problem under consideration in this paper is formally stated.
\begin{Problem}\label{prob.problem}
    Consider a stochastic, nonlinear system \eqref{eq.stochastic_system}, a state observer of the form \eqref{eq.state_observer}, and a goal set $\S_g$ given by \eqref{eq.goal_set}. Design an estimate-feedback controller of the form \eqref{eq.controller} such that, given a specification $p_g^* \in (0,1)$, the set $\S_g$ is rendered FxTS with probability $p_g^*$.
\end{Problem}

\section{Risk-Aware Fixed-Time Stabilization}\label{sec.ra_lfs}

This section introduces two approaches to risk-aware fixed-time stabilization, namely the risk-aware fixed-time control Lyapunov function (RA-FxT-CLF) and the risk-aware path integral control Lyapunov function (RA-PI-CLF), the use of either of which will be shown to render a goal set fixed-time stable with probability $p_g^*$.

First, suppose that there exist $c_1,c_2 > 0$, $\gamma_1 \in (0,1)$, and $\gamma_2 > 1$ for which
\begin{equation}\label{eq.clf_time_constraint}
    T_g \triangleq \frac{1}{c_1(1 - \gamma_1)} + \frac{1}{c_2(\gamma_2 - 1)} \leq T,
\end{equation}
such that $T_g \in \T$ is within the interval of system operation. Let $\eta(\phi, \S): \mathcal{C}^2_0 \times 2^\mathcal{X}$ be a robustness measure defined by
\begin{equation}\label{eq.eta_func}
    \eta(\phi, \S) \triangleq \sup_{\bx \in \S}\left\|\frac{\partial \phi}{\partial \bx}\sigma_x(\bx)\right\|.
\end{equation}
With the state unavailable for measurement, it is not possible to use the generator $\L^{\bx}V$ for control design. Instead, consider the generator of the observer process $\bxh_t$, given by
\small{
\begin{equation*}
\begin{aligned}
    \L^{\bxh}V(\bx,\bxh,\bb{u}) = \frac{\partial V}{\partial \bxh}\big(f(&\bxh, \bb{u}) + \bb{K}(h(\bx) - h(\bxh))\big) \\
    &+ \frac{1}{2}\textrm{Tr}\left(\sigma_y(\bxh)^\top\bb{K}^\top\frac{\partial^2 V}{\partial \bxh^2}\bb{K}\sigma_y(\bxh)\right),
\end{aligned}
\end{equation*}
}\normalsize
and note that it cannot be determined exactly due to $h(\bx)$. To overcome this and to introduce the subsequent lemma, the following assumption is made throughout the remainder.
\begin{Assumption}\label{ass.lipschitz_and_boundedness}
    Over the domain $\X \setminus \S_g$,
    \begin{enumerate}
        \item the functions $f(\cdot,\bb{u}), \forall \bb{u} \in \mathcal{U}$, and $\sigma_x$, $h$, $\sigma_y$, $V$, $\frac{\partial V}{\partial \bb{x}}$, and $\frac{\partial^2 V}{\partial \bx^2}$ are locally Lipschitz with known Lipschitz constants $\lambda_i$ for $i \in \{f, \sigma_x, h, \sigma_y, \partial_{\bx} V, \partial^2_{\bx} V\}$, and
        \item the functions $f(\cdot, \bb{u}), \forall \bb{u} \in \mathcal{U}$, and $\sigma_x$, $\sigma_y$, $\frac{\partial V}{\partial \bx}$, and $\frac{\partial^2 V}{\partial \bx^2}$ are bounded with known bounds $\beta_i$ for $i \in \{f, \sigma_x, \sigma_y, \partial_{\bx} V, \partial^2_{\bx} V\}$.
    \end{enumerate}
\end{Assumption}
\begin{Lemma}\label{lem.Lphi_Lipschitz}
    Suppose that Assumptions \ref{ass.bounded_observer_error} and \ref{ass.lipschitz_and_boundedness} hold.
    Then, 
\begin{equation*}
    \P\left\{|\L^{\bx}V(\bx,\bb{u}) - \bar{\L}^{\bxh}V(\bxh,\bb{u})| \leq \epsilon\lambda_{\L V}, \forall t \in \T\right\} \geq 1 - \delta,
\end{equation*}
and, for any sample path $\omega \in \Omega$,
\begin{equation*}
    \P\left\{|I_V(t,\omega) - \bar I_{\hat V}(t,\omega)| \leq t\epsilon\lambda_{\L V}, \forall t \in \T\right\} \geq 1 - \delta,
\end{equation*}
where
\begin{align}
    \bar{\L}^{\bxh}V(\bxh, \bb{u}) &\triangleq \frac{\partial V}{\partial \bxh}f(\bxh, \bb{u}) + \lambda_h\epsilon\left\|\frac{\partial V}{\partial \bxh}\bb{K}\right\| \label{eq.Lbar_xhat} \\ &\quad\quad\quad\quad + \frac{1}{2}\textrm{Tr}\left(\sigma_y(\bxh)^\top\bb{K}^\top\frac{\partial^2 V}{\partial \bxh^2}\bb{K}\sigma_y(\bxh)\right), \nonumber \\
    \bar I_{\hat V}(t,\omega) &\triangleq \int_0^t\bar{\L}^{\bxh} V(\bxh_s(\omega),\bb{u}_s(\omega))\d s, \label{eq.Ihat}
\end{align}
and
$\lambda_{\L V} = \beta_f\lambda_{\partial_{\bx}V} + \beta_{\partial_{\bx}V}\lambda_f + \beta_{\partial_{\bx}V}\beta_K\lambda_h + \lambda_p + \lambda_q$, with $\lambda_p$ and $\lambda_q$ being the local Lipschitz constants of $p(\bx) = \frac{1}{2}\mathrm{Tr}[\sigma_x(\bx)^\top\frac{\partial^2 V}{\partial\bx^2}\sigma_x(\bx)]$ and $q(\bx) = \frac{1}{2}\mathrm{Tr}[\sigma_y(\bx)^\top\bb{K}^\top\frac{\partial^2 V}{\partial\bx^2}\bb{K}\sigma_y(\bx)]$ respectively.
\end{Lemma}
\begin{proof}
    Follows directly from the application of Assumptions \ref{ass.bounded_observer_error} and \ref{ass.lipschitz_and_boundedness}, and is omitted for brevity.
\end{proof}
The above provides a probabilistic error bound on $|\L^{\bx}V - \bar\L^{\bxh}V|$ using only known quantities, and is instrumental in the derivation of the main results to follow.

\subsection{Risk-Aware Fixed-Time CLF}
We now formally introduce the RA-FxT-CLF and illustrate how its existence certifies a set $\S_g$ as $p_g^*$-FxTS under stochastic dynamics and uncertain measurements.
\begin{Definition}\label{def.ra_fxt_clf}
    Suppose that Assumptions \ref{ass.bounded_observer_error} and \ref{ass.lipschitz_and_boundedness} hold, and consider the set $\S_g$ defined by \eqref{eq.goal_set} for a twice continuously differentiable function $V: \X \mapsto \R$ satisfying \eqref{eq.V_pos_decrescent} and \eqref{eq.V0_r}. The function $V$ is a \textbf{risk-aware fixed-time control Lyapunov function} (RA-FxT-CLF) for the interconnected system $($\eqref{eq.stochastic_system}, \eqref{eq.observer_controller_system}$)$ w.r.t. $\S_g$ if there exists $\hat p_g \in (\pi_g,1)$ such that, on every sample path $\omega \in \Omega$, the following holds $\forall \bxh \in \{\bxh \in \X \mid V_r(\bxh) > 0\}$,
    \begin{equation}\label{eq.rafxtclf_condition}
         \small\inf_{\bb{u} \in \mathcal{U}}\bar{\L}^{\bxh}V(\bxh(\omega),\bb{u}) \leq -c_1V_r(\bxh(\omega))^{\gamma_1} - c_2V_r(\bxh(\omega))^{\gamma_2} - \epsilon\lambda_{\L V},
    \end{equation}
    where $V_r(\bb{z}) \triangleq V(\bb{z}) + r$ for
    \begin{equation}\label{eq.ra_fxt_r}
        r = \sqrt{2T_g}\vartheta\erf^{-1}(2 \hat p_g - 1) + \epsilon\lambda_V,
    \end{equation}
    and $\pi_g = \frac{1}{2}(1 + \erf(\frac{\lambda_V\epsilon}{\vartheta\sqrt{2T_g}}))$, 
    with $\vartheta = \eta(V,\mathcal{O}_{g,r})$ for $\eta$ defined by \eqref{eq.eta_func} and $\mathcal{O}_{g,r} = \{\bx \in \X \mid V(\bx) > -r\}$.
\end{Definition}

\begin{Theorem}\label{thm.ra_fxt_clf}
    If $V$ is a RA-FxT-CLF for the interconnected system $($\eqref{eq.stochastic_system}, \eqref{eq.observer_controller_system}$)$ w.r.t. $\S_g$, then the set $\S_g$ is rendered $p_g$-FxTS with probability $p_g = \hat p_g(1 - \delta)$.
\end{Theorem}
\begin{proof}
    We first provide a derivation for the idealized case ($\bxh \equiv \bx$), and then use it to prove the main result.

    It follows from It$\hat{\mathrm{o}}$'s Formula \cite[Thm. 4.2.1]{Oksendal2003Stochastic} that $\forall t \in \T$,
    \begin{equation}
        \text{d}V(\bx_t) = \L^{\bx} V(\bx_t, \bb{u}_t)\dt+ \frac{\partial V}{\partial \bx}\sigma_x(\bx_t)\text{d}\bb{w}_t, \nonumber
    \end{equation}
    which implies that, for sample path $\omega \in \Omega$, $V(\bx_t(\omega)) \sim V(\bx_0) + I_V(t, \omega) + I_V^S(t, \omega)$, where $I_V(t, \omega)$ is of the form \eqref{eq.integrated_generator} and $I_V^S(t, \omega) \triangleq \int_0^t\frac{\partial V}{\partial \bx}\sigma_x(\bx_s(\omega))\text{d}\bb{w}_s$, which, henceforth omitting $\omega$, is an It$\hat{\mathrm{o}}$ integral \cite[Def. 3.1.6]{Oksendal2003Stochastic} distributed via
    \begin{equation*}
        I_V^S(t) \sim \mathcal{N}\left(0, \; \mathbb{E}\left[\left(\int_0^t\frac{\partial V}{\partial \bx}\sigma_x(\bx_s)\text{d}\bb{w}_s\right)^2\right]\right).
    \end{equation*}
    From Lemma \ref{lem.multidim_ito_isometry}, it follows that
    \footnotesize\begin{equation*}
        \mathbb{E}\left[\left(\int_0^t\frac{\partial V}{\partial \bx}\sigma_x(\bx_s)\text{d}\bb{w}_s\right)^2\right] =\mathbb E\left[\int_0^t\left\|\frac{\partial V}{\partial \bx}\sigma_x(\bx_s)\right\|^2\text{d}s\right] \triangleq s^2(t),
    \end{equation*}\normalsize
    which implies $V(\bx_t) \sim  \mathcal{N}(\mu_V(t), s^2(t))$, where $\mu_V(t) = V(\bx_0) + I_V(t)$. Now, consider a normally distributed random variable 
    \begin{equation*}
        \bar V(\bx_t) \sim V(\bx_0) + I_V(t) + \bar I_V^S(t),
    \end{equation*}
    where $\bar I_V^S(t) \triangleq \vartheta w_t \sim \mathcal{N}(0, \vartheta^2 t)$ for a 1D standard Wiener process $w_t$, noting that $\vartheta^2 t = \int_0^t\vartheta^2 \text{d}s$. With $\int_0^t\vartheta^2\text{d}s \geq s^2(t)$ by construction, it follows from Gaussian properties that
    $\P\{\bar I_V^S(t) > a\} > \P\{I_V^S(t) > a\}$, for all $a, t > 0$. As such, for any $b\geq 0$, $\P\{V_b(\bx_t) \leq 0\} \geq \P\{\bar V_b(\bx) \leq 0\}$, where $V_b(\bb{z}) = V(\bb{z}) + b$, $\bar V_b(\bb{z}) = \bar V(\bb{z}) + b$, and $\L^{\bx} V_b = \L^{\bx} V$.
    
    Now, observe that by \cite[Cor. 3.4]{Yu2019Fixed} it is true that if
    \begin{equation}\label{eq.ideal_ra_fxt_clf_condition}
        \inf_{\bb{u} \in \mathcal{U}}\L^{\bx} V(\bx,\bb{u}) \leq -c_1V_b^{\gamma_1}(\bx) - c_2V_b^{\gamma_2}(\bx), \forall \bx \in \mathcal{X},
    \end{equation}
    then the set $\S_g$ is stochastically FxTS in probability, which implies that $\mathbb{E}[T(\bx_0)] \leq T_g$, $\forall \bx_0 \in \X_0$ with $T_g$ given by \eqref{eq.clf_time_constraint}, i.e., that $\forall \bx_0 \in \X_0$, $\bx_t \rightarrow \S_{g,b} = \{\bb{x} \in \X \mid V(\bx) \leq -b\}$ as $t \rightarrow T_g$ with probability $0.5$. This implies that $V(\bx_{T_g}) \sim \mathcal{N}(-b, s^2(T_g))$ and consequently $\bar V(\bx_{T_g}) \sim \mathcal{N}(-b,\vartheta^2T_g)$. It then follows from the Gaussian cumulative distribution function that 
    \begin{equation*}
        \P\left\{\bar V(\bx_{T_g}) \leq 0\right\} = \frac{1}{2}\left[1 + \erf\left(\frac{b}{\vartheta\sqrt{2T_g}}\right)\right],
    \end{equation*}
    and therefore, choosing $\hat p_g \in (\pi_g,1)$ and solving $\hat p_g = \frac{1}{2}\left[1 + \erf\left(\frac{b}{\vartheta\sqrt{2T_g}}\right)\right]$ for $b$ we obtain that $b = r - \epsilon\lambda_{V}$ with $r$ given by \eqref{eq.ra_fxt_r}. Thus, when \eqref{eq.ideal_ra_fxt_clf_condition} is satisfied it follows that the set $\S_g$ is $\hat p_g$-FxTS, i.e., that $\bb{x}_t \rightarrow \S_g$ as $t \rightarrow T_g$ with probability $\hat p_g$. 
    
    Now, by the Lipschitz property of $V$, consider that $\forall b \geq 0$, $\forall \bx, \bxh \in \X$, $V_b(\bx) \leq V_b(\bxh) + \epsilon\lambda_V$ with probability $1-\delta$. Thus, for \eqref{eq.ideal_ra_fxt_clf_condition} to be satisfied with probability $1-\delta$ the following condition must hold:
    \begin{equation*}
        \inf_{\bb{u} \in \mathcal{U}}\L^{\bx} V(\bx,\bb{u}) \leq -c_1V_r^{\gamma_1}(\bxh) - c_2V_r^{\gamma_2}(\bxh), \forall \bx \in \mathcal{X}.
    \end{equation*}
    Then, by Lemma \ref{lem.Lphi_Lipschitz} it follows that the satisfaction of \eqref{eq.rafxtclf_condition} implies \eqref{eq.ideal_ra_fxt_clf_condition} with probability $1 - \delta$. Therefore, the set $\S_g$ is $p_g$-FxTS with $p_g=(1-\delta)\hat p_g$. This completes the proof.


    
\end{proof}

When the state is known exactly ($\bxh_t \equiv \bx_t$), the only required modification to the RA-FxT-CLF for use in control design is to set $\delta = \epsilon = 0$.

\subsection{Risk-Aware Path Integral CLF}
A potential drawback to using the RA-FxT-CLF given by Definition \ref{def.ra_fxt_clf} for controller design and/or verification is that $\hat p_g \geq \pi_g$, which, for large $\lambda_V$ or $\epsilon$ may produce $\pi_g \rightarrow 1$. This may make \eqref{eq.rafxtclf_condition} difficult to satisfy in practice, especially in the presence of other hard system constraints (like safety). The following notion of the RA-PI-CLF, which is inspired by the risk-aware control barrier functions introduced by \cite{black2023safety} and allows for arbitrary $\hat p_g \in [0.5,1)$ (thus opening up the interval $[0.5, \pi_g)$), helps mitigate this issue.
\begin{Definition}\label{def.ra_pi_clf}
    Suppose that Assumptions \ref{ass.bounded_observer_error} and \ref{ass.lipschitz_and_boundedness} hold, and consider a set $\S_g$ defined by \eqref{eq.goal_set} for a twice continuously differentiable function $V: \X \mapsto \R$ satisfying \eqref{eq.V_pos_decrescent} and \eqref{eq.V0_r}. 
    The function $V$ is a \textbf{risk-aware path integral control Lyapunov function} (RA-PI-CLF) for the interconnected system $($\eqref{eq.stochastic_system}, \eqref{eq.observer_controller_system}$)$ w.r.t. $\S_g$ if there exists $\hat p_g \in [0.5,1)$ such that on every sample path $\omega \in \Omega$, the following holds $\forall t \leq T$,
    \begin{equation}\label{eq.rapiclf_condition}
        \inf_{\bb{u} \in \mathcal{U}}\bar{\L}^{\bxh}V(\bxh_t(\omega),\bb{u}) \leq -c_1W(t, \omega)^{\gamma_1} - c_2W(t, \omega)^{\gamma_2},
    \end{equation}
    where 
    \begin{equation}\label{eq.w_clf}
        W(t, \omega) = \bar I_{\hat V}(t, \omega) + \zeta\sqrt{2T}\erf^{-1}(\hat p_g) + \gamma_V + \epsilon T\lambda_{\L V},
    \end{equation}
    with $\bar I_{\hat V}$ defined according to \eqref{eq.Ihat} and $\zeta = \eta(V,\mathcal{O}_g)$ for $\eta$ defined by \eqref{eq.eta_func}.
\end{Definition}
\begin{Theorem}\label{thm.ra_lf}
    If $V$ is a RA-PI-CLF for the interconnected system $($\eqref{eq.stochastic_system}, \eqref{eq.observer_controller_system}$)$ w.r.t. the set $\S_g$, then the set $\S_g$ is rendered $p_g^*$-FxTS with probability $p_g^* = \hat p_g(1 - \delta)$.
\end{Theorem}
\begin{proof}
    It will be shown that if $V$ is a RA-PI-CLF then $\S_g$ is 1) stable in probability, and 2) locally finite-time attractive with probability $p_g^*$ and uniformly bounded settling time.

    The first components of this proof mirror the proof of Theorem \ref{thm.ra_fxt_clf}, and thus we skip to comparing $\bar V$ and $V$ as follows: 
    $\forall c \geq 0$,
    \begin{equation*}
        \bar p_v(c) \triangleq \P\left\{\sup_{t \in \T}\bar V(\bx) \leq c\right\} \leq \P\left\{\sup_{t \in \T} V(\bx) \leq c\right\} \triangleq p_v(c).
    \end{equation*}
    Now, let $\mathcal B_c = \{\bx \in \X \mid V(\bx) \leq c\}$ and $R(c) = \sup_{\bx \in B_c}\|\bx\|$, and note that $R(c) < \infty$ for $c < \infty$ due to $V$ satisfying \eqref{eq.V_pos_decrescent}. Observe that for a given sample path $\omega \in \Omega$
    \begin{align*}
        \bar p_v(c) &= \P\left\{\sup_{t \in \T} \big[V(\bx_0) + I_V(t,\omega) + \zeta w_t\big] \leq c\right\}, \\
        & \geq \P\left\{\gamma_V + \sup_{t \in\T}I_V(t,\omega) + \sup_{t \in\T}\zeta w_t \leq c\right\}, \\
        &= \P\left\{\sup_{t \in\T}w_t < \frac{c -\gamma_V - \sup_{t \in\T}I_V(t,\omega)}{\zeta }\right\}.
    \end{align*}
    Since $\P\{\sup_{t \in \T}|I_V(t,\omega) - \bar I_{\hat V}(t,\omega)| \leq \epsilon t\lambda_{\L V}\} \geq 1 - \delta$ from Lemma \ref{lem.Lphi_Lipschitz}, it is true that $I_V(t,\omega) \leq \bar I_{\hat V}(t,\omega) + \epsilon t \lambda_{\L V} \leq \bar I_{\hat V}(t,\omega) + \epsilon T \lambda_{\L V}$ with probability of at least $1 - \delta$. Thus,
    \footnotesize{
    \begin{equation*}
        \bar p_v(c) \geq \P\left\{\sup_{t \in\T}w_t < \frac{c -\gamma_V - \sup_{t \in\T}\bar I_{\hat V}(t,\omega) - \epsilon T\lambda_{\L V}}{\zeta }\right\}(1 - \delta).
    \end{equation*}
    }\normalsize
    With $W(t,\omega) \geq 0$ whenever $\bar I_{\hat V}(t,\omega) \geq -\zeta \sqrt{2T}\erf^{-1}(\hat p_g) - \gamma_V - \epsilon T\lambda_{\L V}$, it follows by \eqref{eq.rapiclf_condition} that when $V$ is a RA-PI-CLF, $\L^{\bx} V \leq 0$ whenever $\bar I_{\hat V}(t,\omega) \geq -\zeta \sqrt{2T}\erf^{-1}(\hat p_g) - \gamma_V - \epsilon T\lambda_{\L V}$, and thus $\sup_{t \in \T}\bar I_{\hat V}(t,\omega) = 0$. Then, by Lemma \ref{lem.wiener_crossing_probability},
    \begin{equation*}
        \bar p_v(c) \geq \erf\left(\frac{c -\gamma_V - \epsilon T\lambda_{\L V}}{\zeta \sqrt{2T}}\right)(1 - \delta).
    \end{equation*}
    Taking the above with equality and setting $\varepsilon(c) = 1 - \bar p_v(c)$, it holds that, $\forall R(c) \geq 0$,
    \begin{equation*}
        \P\left\{\sup_{t \in \T}\|\bx(t; \bx_0)\| \leq R(c)\right\} \geq 1 - \varepsilon(c),
    \end{equation*}
    which by Definition \ref{def.stochastically_FTS} implies that $\S_g$ is stable in probability.

    Now, by similar arguments, it follows that
    \begin{equation*}
        \bar p_g \triangleq \P\left\{\inf_{t \in \T}\bar V(\bx) \leq 0\right\} \leq \P\left\{\inf_{t \in \T} V(\bx) \leq 0\right\} \triangleq p_g.
    \end{equation*}
    As such, following the same steps as above (omitting $\omega$),
    \begin{equation*}
        \bar p_g \geq \erf\left(\frac{-\gamma_V - \inf_{0 \leq t \leq T}\bar I_{\hat V}(t) - \epsilon T\lambda_{\L V}}{\zeta \sqrt{2T}}\right)(1 - \delta),
    \end{equation*}
    which implies that $\bar p_g \geq p^*_g = \hat p_g(1 - \delta)$ provided that $\exists \tau \in \T$ for which the following holds, $\forall t \in [\tau,T]$,
    \begin{equation*}
        \bar I_{\hat V}(t) \leq -\zeta \sqrt{2T}\erf^{-1}\big(\hat p_g\big) - \gamma_V - \epsilon T\lambda_{\L V}.
    \end{equation*}
    Therefore, the above is satisfied when the function $W$ given by \eqref{eq.w_clf} satisfies $W(t) \leq 0$, $\forall t \in [\tau, T]$. Note that $\dot W =\bar \L^{\bxh} V(\bxh,\bb{u})$, and so \eqref{eq.rapiclf_condition} is equivalent to $\dot W \leq -c_1W^{\gamma_1} - c_2W^{\gamma_2}$, which by \cite[Lemma 1]{polyakov2011nonlinear} renders $W$ fixed-time stable to the origin, i.e., $W \rightarrow 0$ as $t \rightarrow T(\bx_0) \leq T_g$ given by \eqref{eq.clf_time_constraint} and $W(t) = 0$ for all $t \in [T_g,T]$. Therefore, $\S_g$ is $p_g^*$-FxTS, i.e., finite-time attractive with probability $p_g^* = \hat p_g(1-\delta)$ and uniformly bounded settling time, $T(\bx_0) \leq T_g$.


\end{proof}




\subsection{Control Design for Risk-Aware FxT-Stabilization}

The following remark motivates the class of control laws proposed in this section for risk-aware control design.
\begin{Remark}\label{rem.every_sample_path}
    For a function $V$ to meet the criteria for Definitions \ref{def.ra_fxt_clf} or \ref{def.ra_pi_clf}, the respective condition (i.e., \eqref{eq.rafxtclf_condition} or \eqref{eq.rapiclf_condition}) is required to hold \textnormal{on every sample path $\omega \in \Omega$}. While this may render the problem of a priori system verification using these classes of functions challenging to solve, at runtime either condition (\eqref{eq.rafxtclf_condition} or \eqref{eq.rapiclf_condition}) is only required to hold for the \textnormal{realized} sample path. Thus, in practice RA-FxT-CLFs and RA-PI-CLFs may be used to filter a nominal control policy deployed in real time.
\end{Remark}
In accordance with the above Remark, the RA-FxT-CLF and RA-PI-CLF may be synthesized in the following optimization-based control law:
\begin{subequations}\label{eq.ra_clf_qp}
\begin{align}
    \bb{u}^* = \argmin_{\bb{u} \in \mathcal{U}} \frac{1}{2}\|\bb{u}&-\bb{u}^0\|^2 \label{subeq.qp_objective}\\
    \textrm{s.t.} \quad \nonumber \\
    \inf_{\bb{u} \in \mathcal{U}}\bar{\L}^{\bxh}V(\bxh_t(\omega),\bb{u}) &\leq \xi(t, \bxh(\omega), \omega), \label{subeq.ra_clf_qp_constraints}
\end{align}
\end{subequations}
where $\bb{u}^0 \in \mathcal{U}$ is a nominal or legacy controller such that the cost function \eqref{subeq.qp_objective} seeks to minimize the deviation of the solution $\bb{u}^*$ from the nominal input, and where the function $\xi: \mathcal{T} \times \R^n \times \Omega \mapsto \R$ is such that \eqref{subeq.ra_clf_qp_constraints} represents either \eqref{eq.rafxtclf_condition} (for RA-FxT-CLF) or \eqref{eq.rapiclf_condition} (for RA-PI-CLF). Note that when $f$ in \eqref{eq.stochastic_dynamics} is affine in the control, \eqref{eq.ra_clf_qp} is a quadratic program and thus may be solved efficiently online using commercial or open-source solvers.



\section{Numerical Case Studies}\label{sec.case_studies}
This section demonstrates the use of the proposed controllers via two numerical examples: first, an empirical study on an illustrative nonlinear system highlights the correctness of the theoretical results; and second, a kinematic model of a fixed-wing UAV illustrates their application to scenarios in which safety must be accounted for in control design.

\subsection{Empirical Study: 2D Nonlinear System}
Consider a nonlinear system whose state is $\bx = (x_1, x_2) \in \R^2$, a model for which is given by the following SDE: 
\begin{equation*}
    \begin{bmatrix}
        \mathrm{d}x_1 \\ \mathrm{d}x_2
    \end{bmatrix} = \left(c(\bx)\begin{bmatrix}
        x_2 \\ x_1
    \end{bmatrix} + \frac{1}{c(\bx)}\begin{bmatrix}
        1 & 0 \\ 0 & 1
    \end{bmatrix}\begin{bmatrix}
        u_1 \\ u_2
    \end{bmatrix}\right)\d t + \sigma_x(\bx)\d \bb{w},
\end{equation*}
where $\bb{u} = (u_1,u_2) \in \R^2$ is the control input vector, $c(\bx) = \sqrt{x_1^2 + x_2^2 - r^2}$ with $r = 0.05$, $\sigma_x(\bx) = \mathrm{diag}(2, 2) \in \R^{2\times 2}$, and $\bb{w}$ is a standard Wiener process. Note that here and in what follows subscripts are used to index vector components and not time. The measurement model is
\begin{equation*}
    \begin{bmatrix}
        \mathrm{d}y_1 \\ \mathrm{d}y_2
    \end{bmatrix} = \begin{bmatrix}
        x_1 \\ x_2
    \end{bmatrix} + \sigma_y(\bx)\d \bb{v},
\end{equation*}
where $\sigma_y(\bx) = \mathrm{diag}(0.25, 0.25) \in \R^{2\times 2}$ and $\bb{v}$ is a standard Wiener process independent of $\bb{w}$. The control objective is to render a neighborhood $\mathcal{G}$ of the origin $p_g$-FxTS for various values of $p_g \in [0.5, 1.0)$ for $T_g = 1$, i.e., to drive $\bb{x}(t) \rightarrow \mathcal{G}$ as $t \rightarrow T_g$, independent of $\bb{x}(0)$, with a probability of at least $p_g$, where $\mathcal{G} = \{\bb{x} \in \X \mid V(\bx) \leq 0 \}$ for $V(\bx) = \frac{1}{2}(x_1^2 + x_2^2 - r^2)$ with $r > 0$. 

The system was simulated under RA-FxT-CLF- and RA-PI-CLF-based controllers of the form \eqref{eq.ra_clf_qp} with nominal input $\bb{u}_0 = (0, 0)$, as well as a stochastic FxT-CLF (S-FxT-CLF) controller inspired by \cite{Yu2019Fixed} for comparison. 
Note that in this example, \eqref{eq.ra_clf_qp} is a QP.
The RA-FxT-CLF and RA-PI-CLF controllers underwent 6 trials each consisting of 100 simulations where the goal stabilization probability, i.e., probability of reaching the goal set within time $T_g$, was set to $p_g=0.55, 0.75, 0.95$ and, for each case, was tested once for perfect measurements (state-feedback) and once using an EKF for state estimation (estimate-feedback). 
The S-FxT-CLF controller ($p_g=0.50$) was tested over two trials of 100 simulations: perfect measurements and EKF state estimation.
Control parameters were constant over each set of 100 simulations, but varied from trial to trial. 
At the start of each simulation, the initial conditions were randomized according to $\bb{x}(0) = (r_0 \cos\theta_0, r_0\sin\theta_0)$, where $r_0 = \mathbb{U}[1, \sqrt{2}]$ and $\theta_0 = \frac{\pi}{4} + \frac{\pi}{2}q + \mathbb{U}[-\frac{\pi}{8},\frac{\pi}{8}]$, with $q$ sampled from the discrete uniform distribution over the set $\{1, 2, 3, 4\}$.

The empirical results for the S-FxT-CLF, RA-FxT-CLF, and RA-PI-CLF controllers may be found in Tables \ref{tab:sclf_results}, \ref{tab:rafxtclf_results}, and \ref{tab:rapiclf_results} respectively. 
States and state estimates, control inputs, and values for $V(\bx)$ are also displayed in Figures \ref{fig.state_paths}, \ref{fig.control_inputs}, \ref{fig.lyapunov_vals} for the RA-FxT-CLF $p_g=0.55$ trial. 
The proposed controllers satisfied their theoretical probabilistic FxTS bounds in all trials. 
What is also apparent, however, is that decreasing the tolerable risk of failure (i.e., increasing the probability of reaching the goal) generally leads to lower average convergence times, which intuitively suggests that controllers with a stronger aversion to risk seek their goals more aggressively. 
In addition, average time-to-goal increases in the presence of imperfect measurements, which, despite being compensated for with state estimation, increase levels of uncertainty in the system.
Further, the success rate was near 100$\%$ in all cases, which may point to the inherent conservatism required for the proposed methods, which rely on worst-case assumptions both with respect to the effect of the stochastic perturbation on the dynamics and the measurement uncertainty. 
Reducing this conservatism will be the focus of future work. It is evident, however, that by accepting less risk of failure our methods still converge faster to the goal on average than the S-FxT-CLF-based controller.

\begin{table}[t]
    \centering
    \begin{tabular}{c|c|c|c|c}
         Theoretical $p_g$&  Observed $p_g$&  Perfect&  EKF & $T_{avg}$\\
         \hline
         0.50 & 1.00 & x &  & 0.17\\
         \hline
         0.50 & 1.00 &  & x & 0.20\\
    \end{tabular}
    \caption{S-FxT-CLF results from the 2D nonlinear system goal-reaching study over 100 trials.}
    \label{tab:sclf_results}
\end{table}
\begin{table}[t]
    \centering
    \begin{tabular}{c|c|c|c|c|c}
         Theoretical $p_g$&  Observed $p_g$&  Perfect&  EKF & $T_{avg}$\\
         \hline
         0.55 & 1.00 & x &  & 0.12\\
         \hline
         0.75 & 1.00 & x &  & 0.05\\
         \hline
         0.95 & 1.00 & x &  & 0.04\\
         \hline
         0.55 & 1.00 &  & x & 0.27\\
         \hline
         0.75 & 1.00 &  & x & 0.11 \\
         \hline
         0.95 & 1.00 &  & x & 0.14 \\
    \end{tabular}
    \caption{RA-FxT-CLF results from the 2D nonlinear system goal-reaching study over 100 trials.}
    \label{tab:rafxtclf_results}
\end{table}
\begin{table}[t]
    \centering
    \begin{tabular}{c|c|c|c|c|c}
         Theoretical $p_g$&  Observed $p_g$&  Perfect&  EKF & $T_{avg}$\\
         \hline
         0.55 & 1.00 & x &  & 0.06 \\
         \hline
         0.75 & 1.00 & x &  & 0.05 \\
         \hline
         0.95 & 1.00 & x &  & 0.03 \\
         \hline
         0.55 &  0.97 &  & x & 0.14 \\
         \hline
         0.75 &  0.97 &  & x & 0.11 \\
         \hline
         0.95 & 0.97 &  & x & 0.10 \\
    \end{tabular}
    \caption{RA-PI-CLF results from the 2D nonlinear system goal-reaching study over 100 trials.}
    \label{tab:rapiclf_results}
\end{table}

\begin{figure}[!ht]
    \centering
        \includegraphics[width=0.95\columnwidth,clip]{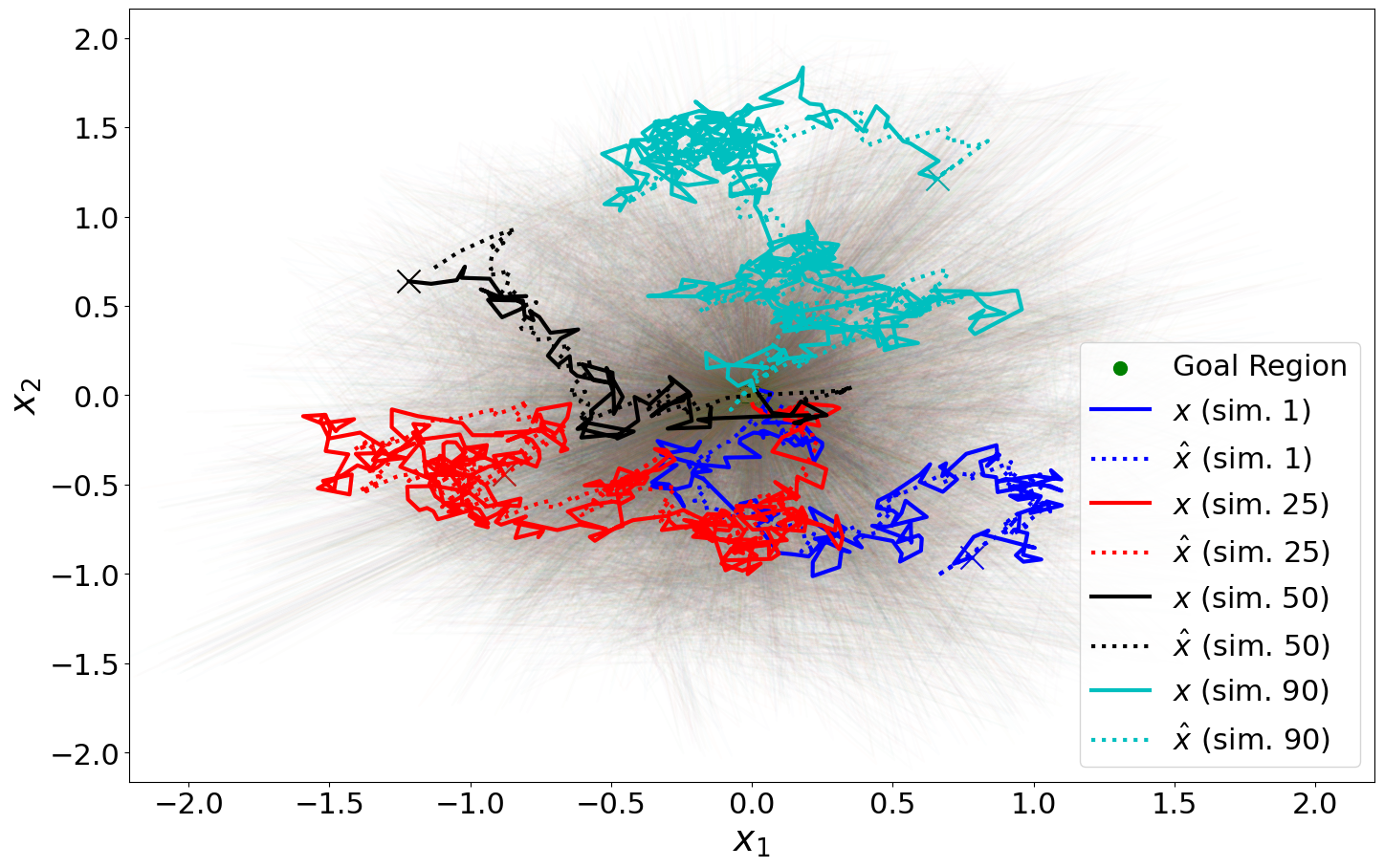}
    \caption{State paths ($x$) starting from initial condition $x_0$ (marked by X) for the RA-FxT-CLF controlled system with EKF state estimation ($\hat x$) and $p_g=0.55$. All paths displayed (translucent), with a selection highlighted (sims. 1, 25, 50, 90).}\label{fig.state_paths}
\end{figure}
\begin{figure}[!ht]
    \centering
        \includegraphics[width=0.95\columnwidth,clip]{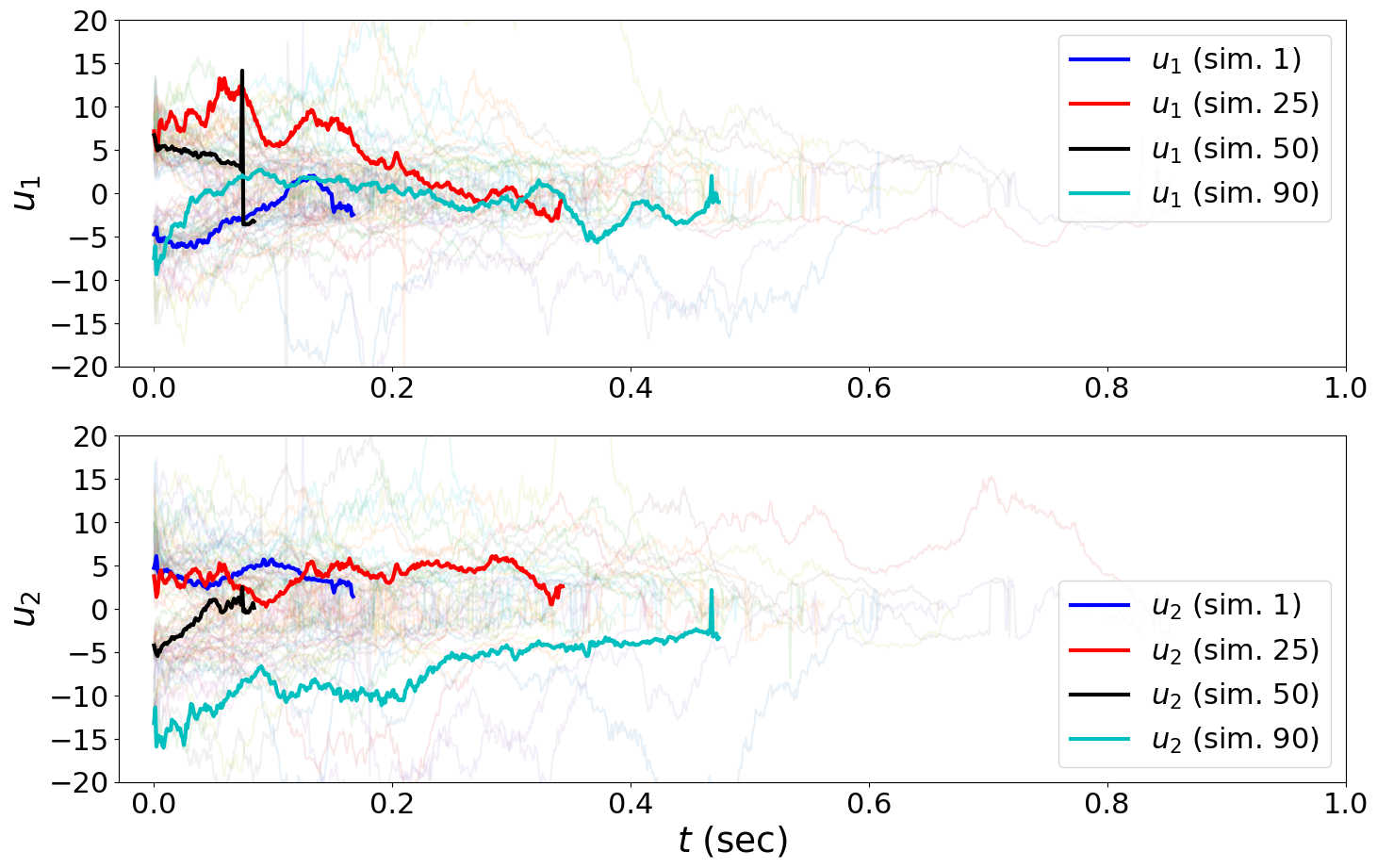}
    \caption{Control trajectories for the RA-FxT-CLF controlled system with EKF state estimation and $p_g=0.55$. All controls displayed (translucent), with a selection highlighted (sims. 1, 25, 50, 90).}\label{fig.control_inputs}
\end{figure}
\begin{figure}[!ht]
    \centering
        \includegraphics[width=0.95\columnwidth,clip]{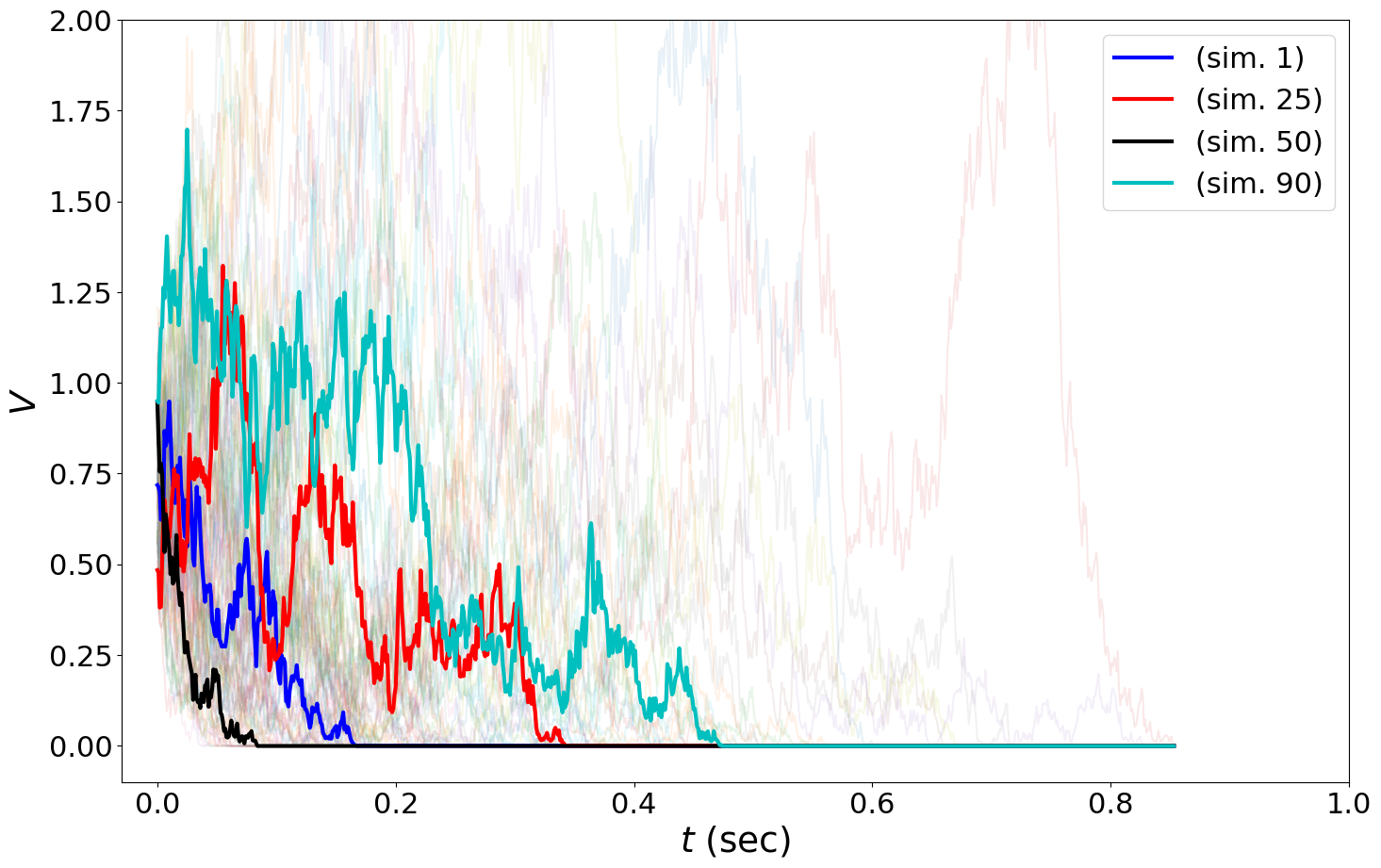}
    \caption{Lyapunov function values of the true state ($V(\bx)$) versus time for the RA-FxT-CLF controlled system with EKF state estimation and $p_g=0.55$. All runs displayed (translucent), with a selection highlighted (sims. 1, 25, 50, 90).}\label{fig.lyapunov_vals}
\end{figure}

\subsection{Fixed-Wing UAV Example}

In this case study on the use of RA-FxT-CLFs for fixed-wing UAV control for reaching a goal set within a fixed-time while avoiding an unsafe set, the following kinematic model inspired by \cite{beard2014fixed} is used to describe the aircraft:

\small{
\begin{equation*}
    \begin{bmatrix}
    \d p_n \\ \d p_e \\ \d h \\ \d v \\ \d \psi \\ \d \gamma
    \end{bmatrix} = \left(\begin{bmatrix}
        v \cos\psi \cos\gamma \\
        v \sin\psi \cos\gamma \\
        v \sin\gamma \\
        0 \\
        0 \\
        0
    \end{bmatrix} + \begin{bmatrix}
        0 & 0 & 0 \\
        0 & 0 & 0 \\
        0 & 0 & 0 \\
        u_1 & 0 & 0 \\
        0 & u_2 & 0 \\
        0 & 0 & u_3 \\
    \end{bmatrix}\right)\d t + \sigma_x(\bx)\d \bb{w},
\end{equation*}}\normalsize
where $\bx = (x, y, z, v, \psi, \gamma) \in \R^6$ denotes the state vector, $\bb{u} = (u_1,u_2,u_3) \in \R^3$ the control input vector, with $p_n$ and $p_e$ the inertial north and east positions, $h$ the altitude, $v$ the airspeed, $\psi$ the heading angle measured from north, $\gamma$ the flight path angle, and where $u_1=a$ is the rate of change of the airspeed, $u_2=\frac{g}{v}\tan\phi$ with $\phi$ the roll angle and $g$ the acceleration due to gravity, and $u_3=\omega$ the rate of change of flight path angle. The stochastic term $\sigma_x(\bx) = \text{diag}(0, 0, 0, w_n, w_e, w_h) \in \R^{6\times 6}$ is meant to represent the uncertain effect of wind perturbing the system.

The control law used for this study was the RA-FxT-CLF-based controller of the form \eqref{eq.ra_clf_qp}, and the objective was to render $(0.75)$-FxTS a goal set defined according to \eqref{eq.goal_set}, where $V(\bx) = V_0(\bx) / \left(1 - \sum_{i=1}^3\left(\frac{1}{b_i(\bx)}\right)\right)$, where $V_0(\bx)=(v\cos\psi\cos\gamma +100)^2 + (v\sin\psi\cos\gamma - \dot y_d)^2 + (v\sin\gamma - \dot z_d)^2$ with $\dot y_d = y_g-y$, $\dot z_d = z_g-z$, and $b_i(\bx) = \dot h_i(\bx) + \alpha(h_i(\bx))$ for $h_i(\bx) = (\frac{x - c_{x,i}}{a_x})^2 + (\frac{y - c_{y,i}}{a_y})^2 + (\frac{z - c_{z,i}}{a_z})^2 -1$ and $\alpha \in \mathcal{K}_\infty$ such that $\{\bx \in \X \mid h_i(\bx) < 0\}$ denotes the undesirable set of states inside the i$^{th}$ ellipsoid obstacle. As such, $V$ is defined similar to classes of barrier-Lyapunov functions (e.g., \cite{romdlony2016stabilization}), which seek to guide the system to the goal set while avoiding the undesirable set.

It may be seen from the XY and YZ planes depicted in Figures \ref{fig.fixed_wing_xy_paths} and \ref{fig.fixed_wing_yz_paths} that the system manages to do exactly this, while Figure \ref{fig.fixed_wing_controls} shows the control inputs used to achieve this result. In particular, Figure \ref{fig.fixed_wing_yz_paths} shows the effect of the controller steering the vehicle further away from the ellipsoid. While perhaps requiring more care and attention to implement, this study highlights how the classes of RA-FxT-CLFs and RA-PI-CLFs may be used for risk-aware control in the presence of state constraints, similar to the classes of barrier-Lyapunov functions mentioned previously.

\begin{figure}[!ht]
    \centering
        \includegraphics[width=0.95\columnwidth,clip]{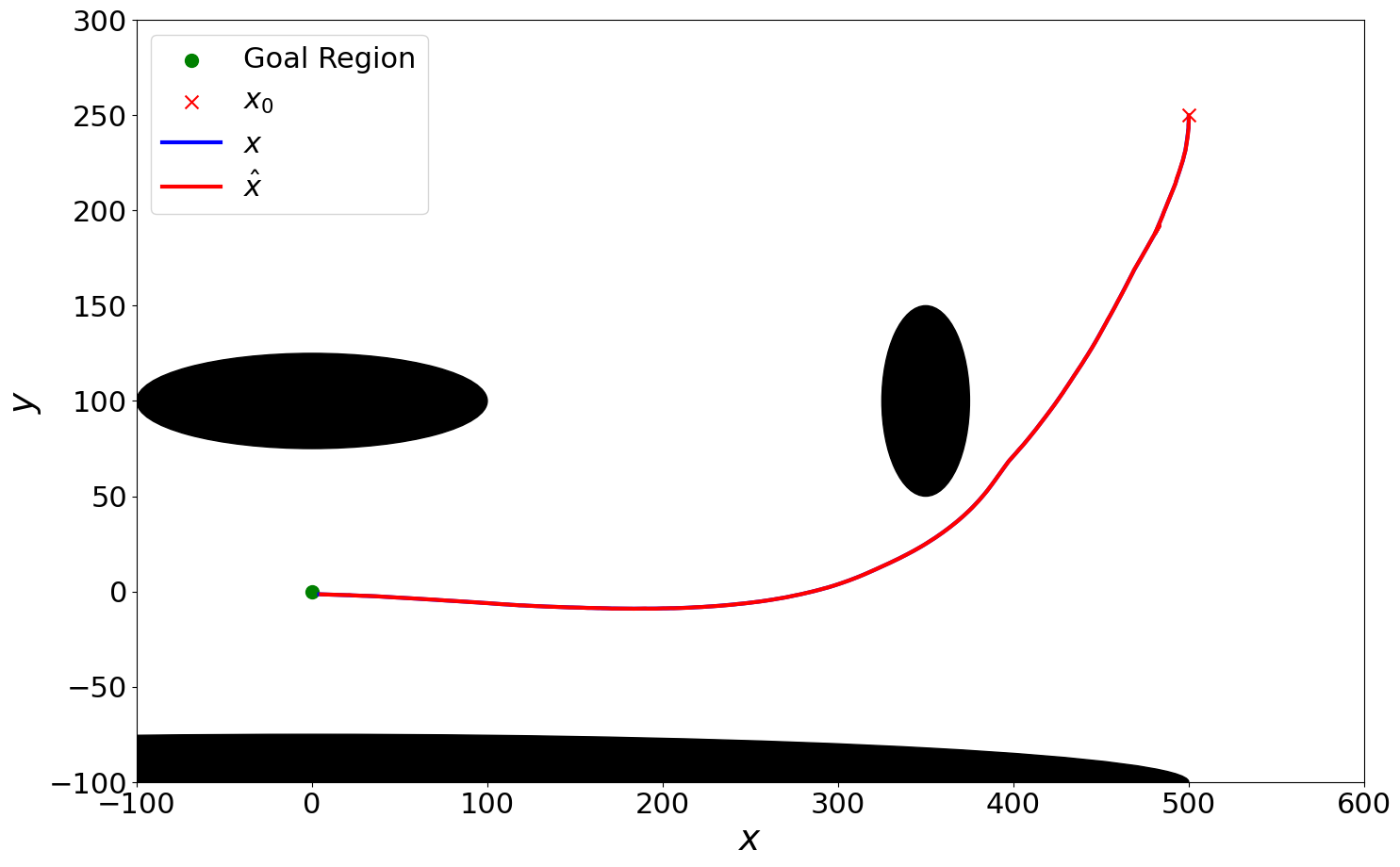}
    \caption{Fixed-wing UAV paths in the XY plane starting from initial condition (marked by X).}\label{fig.fixed_wing_xy_paths}
\end{figure}
\begin{figure}[!ht]
    \centering
        \includegraphics[width=0.95\columnwidth,clip]{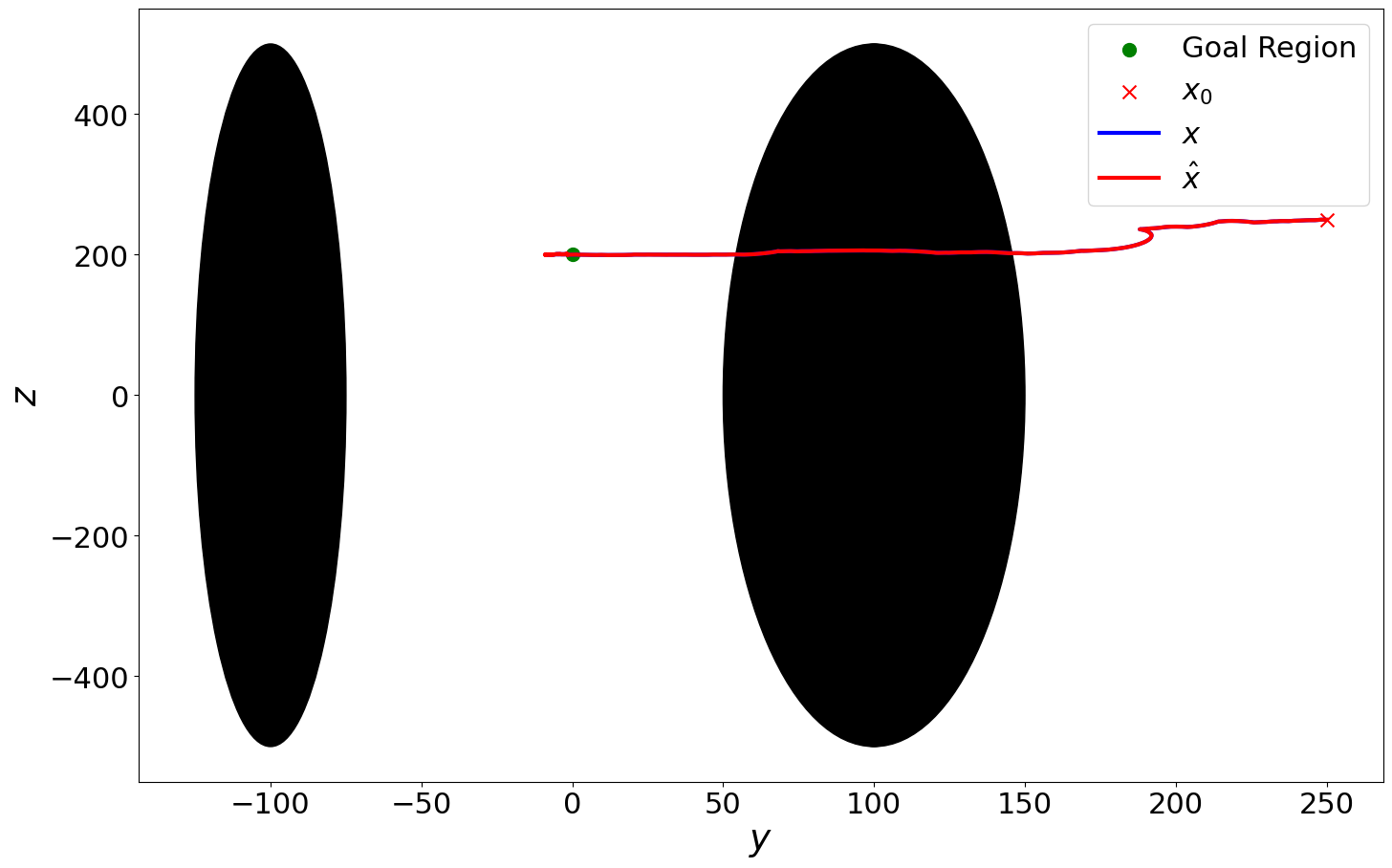}
    \caption{Fixed-wing UAV paths in the YZ plane starting from initial condition (marked by X). Note from Figure \ref{fig.fixed_wing_xy_paths} that the vehicle actually avoids the shown obstacles.}\label{fig.fixed_wing_yz_paths}
\end{figure}
\begin{figure}[!ht]
    \centering
        \includegraphics[width=0.95\columnwidth,clip]{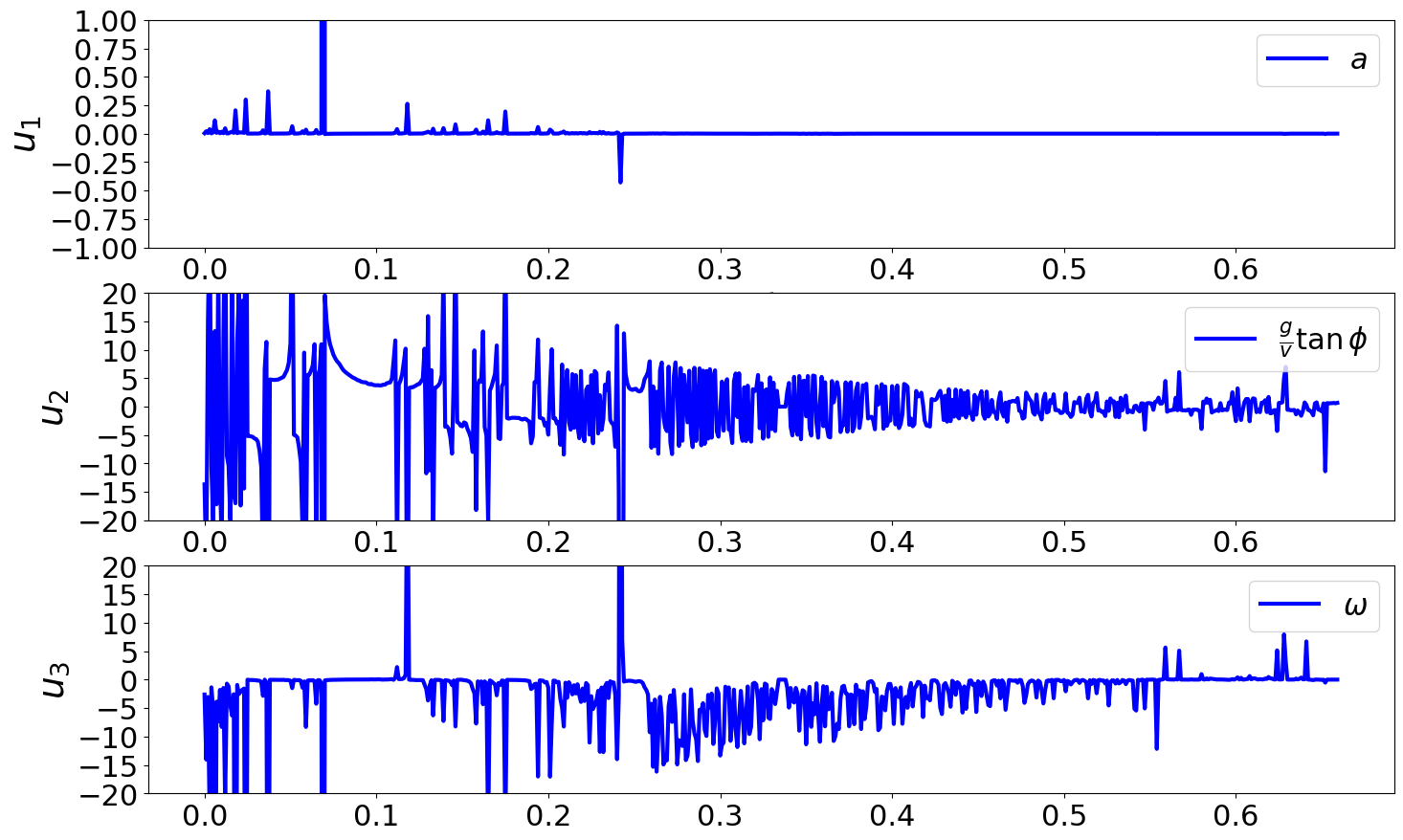}
    \caption{Fixed-wing UAV control trajectories.}\label{fig.fixed_wing_controls}
\end{figure}



\section{Conclusion}\label{sec.conclusion}
In this paper, two Lyapunov-based methods for risk-aware fixed-time stabilization were proposed for a class of stochastic, nonlinear systems subject to measurement uncertainty. It was shown that the use of either approach for control design renders a goal set probabilistically FxTS with probability $p_g$. An empirical study on an illustrative nonlinear system validated the proposed approach in simulation, and a fixed-wing UAV example further demonstrated the efficacy of RA-FxT- and RA-PI-CLFs.

In the future, we plan to explore ways to reduce conservatism associated with the proposed approaches, to combine them with risk-aware barrier functions \cite{black2023safety}, and to investigate their use for a priori system verification.


\bibliographystyle{IEEEtran}
\bibliography{library}


\appendices
\section{Proof of Lemma \ref{lem.wiener_crossing_probability}}\label{app.wiener_crossing_probability}
This proof follows from \cite[Sec. 3]{Blake1973level-crossing}. The cumulative distribution function (CDF) of the supremum of a scalar process $\{x_t: t \in [0, \infty)\}$ over an interval $[0, T]$ is
\begin{equation}
    F_x(a, T) = P\left(\sup_{0 \leq t \leq T}x(t) < a\right) = \int_T^\infty q_a(\theta \; | \; x_0) d\theta, \nonumber
\end{equation}
where $q_a$ is the first passage-time probability density function with respect to level $a>0$. For a standard Wiener process\footnote{A discussion on how to obtain $q_a$ for $w_t$ may be found in \cite{Blake1973level-crossing}.} $\{w_t: t \in [0,\infty)\}$, the CDF of the supremum is given by
\begin{equation*}
    F_w(a, T) = \sqrt{\frac{2}{\pi}}\int_0^\frac{a}{\sqrt{T}}e^{-\frac{s^2}{2}}ds = \mathrm{erf}\left(\frac{a}{\sqrt{2T}}\right).
\end{equation*}

\section{Proof of Lemma \ref{lem.multidim_ito_isometry}}\label{app.proof_multidim_ito_isometry}
Inspired by \cite[Lem. 18]{zhang2020wasserstein}, by the vector dot product it follows that 
\begin{align*}
    \mathbb E\left[\left(\int_0^t\bb{v}_s^\top\text{d}\bb{w}_s\right)^2\right] &= \mathbb E\left[\left(\sum_{i=1}^q\int_0^tv_i\cdot \text{d}w^i_s\right)^2\right].
\end{align*}
Then, since $w_t^i$ and $w_t^j$ are independent for all $i \neq j$, it follows that
\small{
\begin{align*}
    \mathbb{E}\left[\left(\int_0^tv_i\text{d}w^i_s\right) \left(\int_0^tv_j\text{d}w^j_s\right)\right] &= \mathbb{E}\left[\int_0^tv_i\text{d}w^i_s\right]\mathbb{E}\left[\int_0^tv_j\text{d}w^j_s\right] 
    \\ &= 0
\end{align*}
}\normalsize
for all $i \neq j$. Therefore,
\begin{align*}
    \mathbb E\left[\left(\sum_{i=1}^q\int_0^tv_i\cdot \text{d}w^i_s\right)^2\right] &= \mathbb E\left[\sum_{i=1}^q\left(\int_0^tv_i\text{d}w^i_s\right)^2\right], \\
    &= \mathbb E\left[\sum_{i=1}^q\int_0^tv_i^2\dt\right],
\end{align*}
where the last equality follows from the $1$D It$\hat{\mathrm{o}}$ isometry \cite[Lem. 3.1.5]{Oksendal2003Stochastic}). The result then follows directly by definition of the vector 2-norm.

\end{document}